\documentclass{aims}
\usepackage{amsmath}
\usepackage{paralist}
\usepackage{graphics}
\usepackage{epsfig}
\usepackage{graphicx} \usepackage{epstopdf}
\usepackage[colorlinks=true]{hyperref}
\hypersetup{urlcolor=blue, citecolor=red}
\usepackage{setspace,amssymb,color}

\def\DOI{}

\newtheorem{theorem}{Theorem}
\newtheorem{corollary}[theorem]{Corollary}

\newtheorem{lemma}[theorem]{Lemma}
\newtheorem{proposition}[theorem]{Proposition}

\theoremstyle{definition}

\newcommand{\R}{{\mathbb R}}

\renewcommand{\(}{\left(}
\renewcommand{\)}{\right)}

\newcommand{\be}[1]{\begin{equation}\label{#1}}
\newcommand{\ee}{\end{equation}}

\title[Multiplicity results for Fermi--Dirac statistics in gravitation]{Bifurcation diagrams and multiplicity for nonlocal elliptic equations modeling gravitating systems based on Fermi--Dirac statistics}
\author[Jean Dolbeault and Robert Sta\'nczy]{}



\email{dolbeaul@ceremade.dauphine.fr}
\email{stanczr@math.uni.wroc.pl}

\thanks{This work has been partially supported by the Polish Ministry of Science project N N201 418839 and ANR projects STAB and Kibord.}
\RequirePackage[normalem]{ulem} 
\RequirePackage{color}\definecolor{RED}{rgb}{1,0,0}\definecolor{BLUE}{rgb}{0,0,1} 

\begin{document}
\maketitle

\centerline{\scshape Jean Dolbeault}
\medskip{\footnotesize\centerline{Ceremade (UMR CNRS no. 7534), Universit\'e Paris-Dauphine}
\centerline{Place de Lattre de Tassigny, 75775 Paris C\'edex~16, France.}}

\medskip

\centerline{\scshape Robert Sta\'nczy}
\medskip{\footnotesize\centerline{Instytut Matematyczny, Uniwersytet Wroc\l awski}
\centerline{pl. Grunwaldzki 2/4, 50-384 Wroc\l aw, Poland.}}

\bigskip

\thispagestyle{empty}

\begin{abstract} This paper is devoted to multiplicity results of solutions to nonlocal elliptic equations modeling gravitating systems. By considering the case of Fermi--Dirac statistics as a singular perturbation of Maxwell--Boltzmann statistics, we are able to produce multiplicity results. Our method is based on cumulated mass densities and a logarithmic change of coordinates that allow us to describe the set of all solutions by a non-autonomous perturbation of an autonomous dynamical system. This has interesting consequences in terms of bifurcation diagrams, which are illustrated by some numerical computations. More specifically, we study a model based on the Fermi function as well as a simplified one for which estimates are easier to establish. The main difficulty comes from the fact that the mass enters in the equation as a parameter which makes the whole problem non-local.\end{abstract}

\section{Setting of the problem}\label{Sec:Setting}

We consider a stationary solution of the drift-diffusion equation
\be{Eqn:Evol}
\rho_t=\nabla\cdot\left[\rho\(\nabla H(\rho)+\nabla\phi\)\right]
\ee
with a nonlinear diffusion based on some function $H$, coupled with the gravitational Poisson equation
\[
\Delta\phi=\rho
\]
on a bounded domain in $\mathbb{R}^3$, under appropriate boundary conditions. The simplest case, that we will call the \emph{(MB) case}, corresponds to
\[
H(\rho)=\log\rho
\]
for \emph{Maxwell--Boltzmann statistics}, yielding classical linear diffusion in~\eqref{Eqn:Evol}. In this paper we shall study the nonlinear diffusion corresponding to \emph{Fermi--Dirac statistics}, \emph{cf.}~\cite{MR2099972, MR2143357, biler2004parabolic, stanczy-evolution}. In that case the function $H$ is given by
\[
H(\rho)=f_{1/2}^{-1}(2\rho/\mu)
\]
where
\[
f_{1/2}(z):=\int_0^\infty\frac{\sqrt x}{1+\exp(x-z)}\;dx
\]
is a Fermi function. See Appendix~\ref{Appendix} for more details. We shall refer to this case as the (FD) case. In order to deal with more explicit estimates, we shall also consider a \emph{simplified Fermi--Dirac} model. This (sFD) case captures the asymptotic behavior of (FD) and is given by
\[
H(\rho)=\log\rho+\frac32\,\eta\,\rho^{2/3}\;.
\]
Connection between (FD) and (sFD) cases can be established when the positive parameters $\eta$ and $\mu$ are such that
\be{Eqn:EtaMu}
\mu^2\,\eta^3=\frac83\;.
\ee
We again refer to Appendix~\ref{Appendix} for further details. The (FD) model was introduced by P.-\,H.~Chavanis \emph{et al.}~in~\cite{MR2092680, CSR} in the context of astrophysical models of gaseous stars (also see \cite{C} for a general review of the subject). It can be seen as a perturbation of the (MB) model, or linear diffusion model, and we shall prove that some features of the set of the stationary solutions are shared by the linear (MB) model and the nonlinear models (FD) and (sFD), if we take $\eta>0$ sufficiently small (or $\mu>0$ sufficiently large).

We have several reasons to consider the (sFD) case: it has all qualitative features of the (FD) model, equivalents in the asymptotic regimes are much easier to control and numerically it avoids painful computations of Fermi functions, without loosing anything at the level of the mathematical results (qualitative behavior of the solutions) and their physical interpretation.

\medskip In order to use the \emph{cumulated mass} technique, we shall assume that the domain under consideration is the unit ball $B:=\{x\in\R^3\,:\,|x|<1\}$ and that the above equations are respectively supplemented with no-flux boundary conditions
\[
\big(\nabla H(\rho)+\nabla\phi\big)\cdot n=0\quad\mbox{on}\quad\partial B
\]
for the mass density (here $n(x)=x/|x|$ for any $x\in\partial B$) and homogenous Dirichlet boundary conditions for the potential
\[
\phi=0\quad\mbox{on}\quad\partial B\;.
\]
We are interested in the stationary problem with a fixed \emph{mass constraint}
\[
\int_B\rho\;dx=M
\]
that can be solved by
\[
\rho=F\(\lambda-\phi\)
\]
for some appropriate Lagrange multiplier $\lambda$. Since $F$ in all considered cases is monotone increasing, the multiplier $\lambda$ is uniquely defined. Here $F=H^{-1}$ is the inverse of $H$, that is,
\[
F(z)=\frac\mu2\,f_{1/2}(z)
\]
in case (FD) of Fermi--Dirac statistics, and
\[
F(z)=e^z
\]
in case of (MB) Maxwell--Boltzmann statistics. In the (MB) case, the Lagrange multiplier is explicitly given by
\[
e^{\lambda}=\frac{M}{\int_B e^{-\phi}\;dx}\;.
\]
The function $F$ has no simple expression in the (sFD) case. In all cases the stationary problem that we have to solve can be formulated as
\be{Eqn:Fermi--Dirac}
\Delta\phi=F\(\lambda-\phi\)\;\;\mbox{on}\;\; B\;,\quad\phi=0\;\mbox{on}\;\;\partial B\;,\quad\int_BF\(\lambda-\phi\)\;dx=M\;.
\ee

\medskip Let us start with a result in the (MB) case, which is easy to visualize on the bifurcation diagram expressing the dependence of the supremum norm of $\phi$ on the mass parameter. This result is a simple reformulation of a former result from~\cite{BDEMN}. The corresponding diagram exhibits a spiraling structure that can be seen in Fig.~\ref{Fig:Max} (left) and accounts for the multiplicity of solutions which is reflected by the oscillating behavior of the branch in the bifurcation diagram: see Fig.~\ref{Fig:Max} (right).
\begin{proposition}\label{Prop:MB} There exist two sequences $(M_*^n)_{n\ge1}$ and $(M^*_n)_{n\ge1}$ which are ordered and monotone: $M_*^1<M_*^2<\ldots<M^*_2<M^*_1$, such that for any $M\in(M_*^n,M_*^{n+1})\cup(M^*_{n+1},M^*_n)$ there exist at least $2n$ solutions of \eqref{Eqn:Fermi--Dirac} for $F(z)=\exp z$, that is, in the Maxwell--Boltzmann {\rm (MB)} case.\end{proposition}
\begin{proof}[Sketch of the proof] We start by reducing \eqref{Eqn:Fermi--Dirac}, written with $F(z)=e^z$ to the autonomous, dynamical system
\be{Syst:Auton}\left\{\begin{array}{l}
x'=y-\,x\\[6pt]
y'=(2-\,x)\,y
\end{array}\right.\ee
and look for the branch of solutions $s\mapsto(x(s),y(s))$ such that $\lim_{s\to-\infty}(x(s),y(s))\!=(0,0)$, where $x$ and $y$ are defined, consistently with notation in (5) and (7) to be specified later, as
\[
x(\log r)=\frac1r\int_0^rs^2\,e^{\lambda-\phi(s)}\;ds\quad\mbox{and}\quad y(\log r)=r^2\,e^{\lambda-\phi(r)}\quad\forall\,r>0\;.
\]
Then there exists a unique heteroclinic orbit joining the points $(0,0)$ and $(2,2)$ as shown in Fig.~\ref{Fig:Max}. This heteroclinic orbit can be used to parametrize all solutions to~\eqref{Eqn:Fermi--Dirac}.
\begin{figure}[ht]
\label{Fig:Max}
\begin{center}
\includegraphics[height=3.5cm]{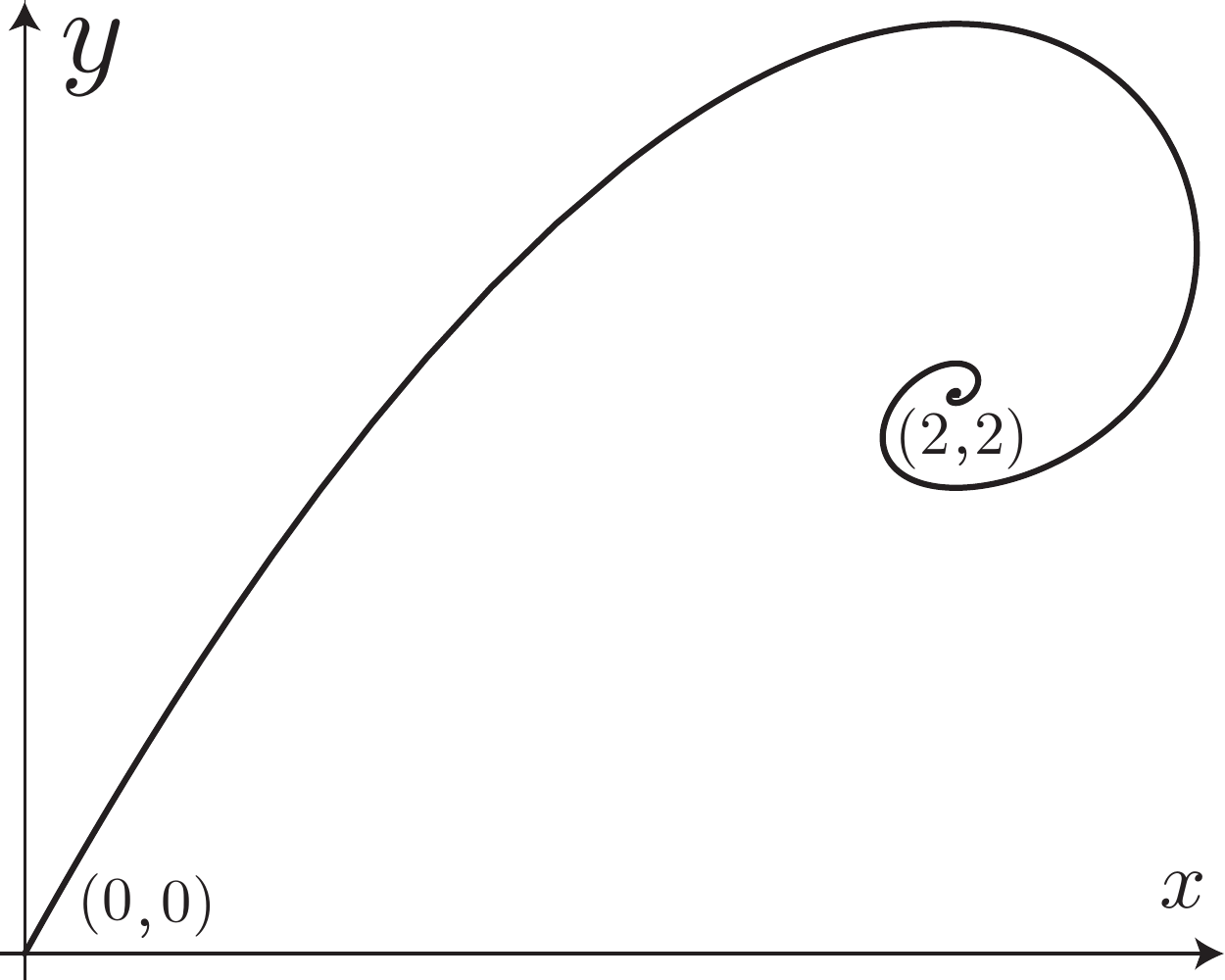} \hspace*{30pt}
\includegraphics[height=3.5cm]{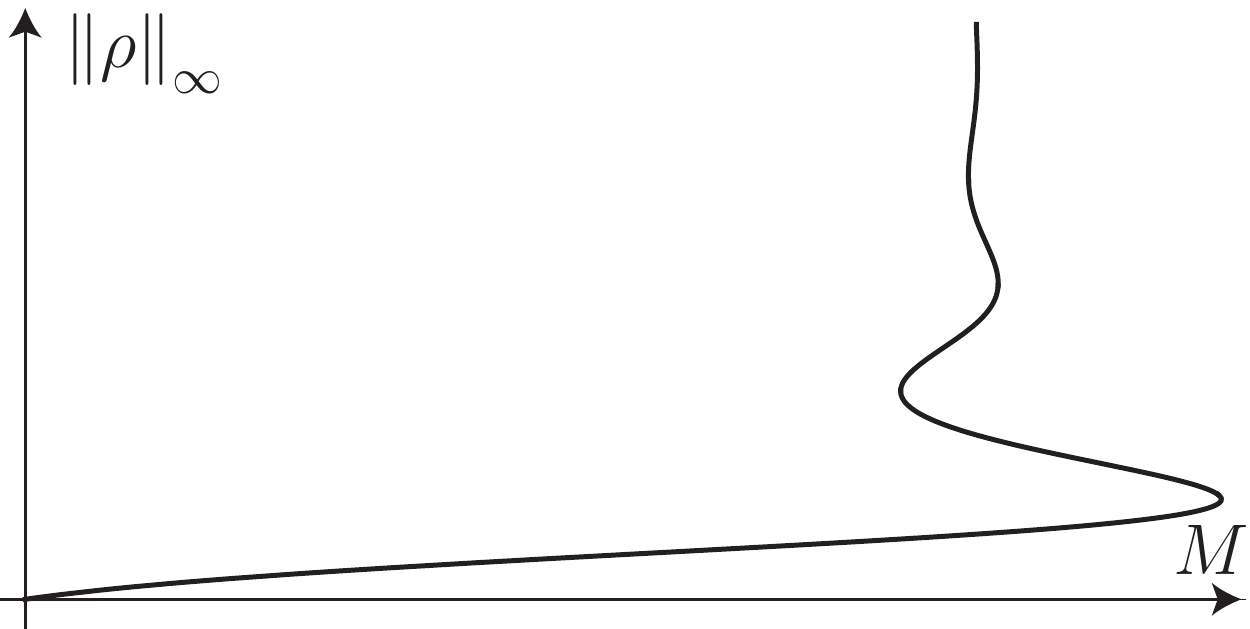}
\end{center}
\caption{\it\small Left: the heteroclinic orbit joining the points $(0,0)$ and $(2,2)$ in the {\rm (MB)} case. Right: the corresponding bifurcation diagram.}
\end{figure}

Since $s\mapsto(x(s+s_0),y(s+s_0))$ is also a solution to \eqref{Syst:Auton}, we may choose $s_0\in\R$ such that $4 \pi\,x(s_0)=M$ at least if $M$ is in the admissible range $4\pi\,x(\R)$ and thus get a solution with $e^\lambda=y(s_0)$. The parameters $M_*^n$ and $M^*_n$ correspond to the lower and upper values of $M$ at which the bifurcation curve turns.\end{proof}

The proof of Proposition~\ref{Prop:MB} is standard in the theory of gravitating systems. The reader is invited to check that all solutions are indeed described by the scheme and is invited to refer to \cite{BDEMN} for more details. Numerically, the scheme allows one to compute all solutions by solving a simple ODE problem. It is also at the core of our approach for the nonlinear case and this is what we are now going to explain.

The main result of this paper is the following theorem. It is a multiplicity result for the Fermi--Dirac model (FD) and its simplified version (sFD) stating that for some values of the mass parameter we have at least the same number of solutions as for the Maxwell--Boltzmann case, if the parameter $\eta>0$ is small enough. We consider the two sequences $(M_*^n)_{n\ge1}$ and $(M^*_n)_{n\ge1}$ that have been defined in Proposition~\ref{Prop:MB} in the (MB) case.
\begin{theorem}\label{Thm:Main} For any $M\in(M_*^n,M_*^{n+1})\cup(M^*_{n+1},M^*_n)$, $n\in \mathbb{N}$, if $\eta>0$ is sufficiently small, there are at least $2n$ solutions of \eqref{Eqn:Fermi--Dirac} in the {\rm (FD)} and {\rm (sFD)} cases.\end{theorem}

In other words, the model corresponding to the Fermi--Dirac statistics or its simplified version at least partially inherits the spiraling structure of the model corresponding to the Maxwell--Boltzmann. See Figs.~\ref{Fig:NonAutonmousPhase} (left) and \ref{Fig:mvs0}. From the numerical computations it can easily be conjectured that a more precise description of the solution set can be achieved, with exact multiplicity. Hence it seems that for $M\in(M_*^n,M_*^{n+1})$ and for $M\in(M^*_{n+1},M^*_n)$ the exact multiplicity are respectively $4n+1$ and $4n-1$, in the asymptotic regime corresponding to $\eta\to0_+$. However, such a conjecture requires a by far more delicate analysis than the one we have done in this paper and is therefore still open. This can be summarized in the bifurcation diagrams for the simplified Fermi--Dirac and various values of $\eta$ approaching $0$. See Figs.~\ref{Fig:NonAutonmousPhase} (right) and~\ref{Fig:mvs}. More details on bifurcation diagrams and further numerical results will be given in Section~\ref{Sec:Bifurcation}.

\begin{figure}[ht]
\label{Fig:NonAutonmousPhase}
\begin{center}
\includegraphics[height=5cm]{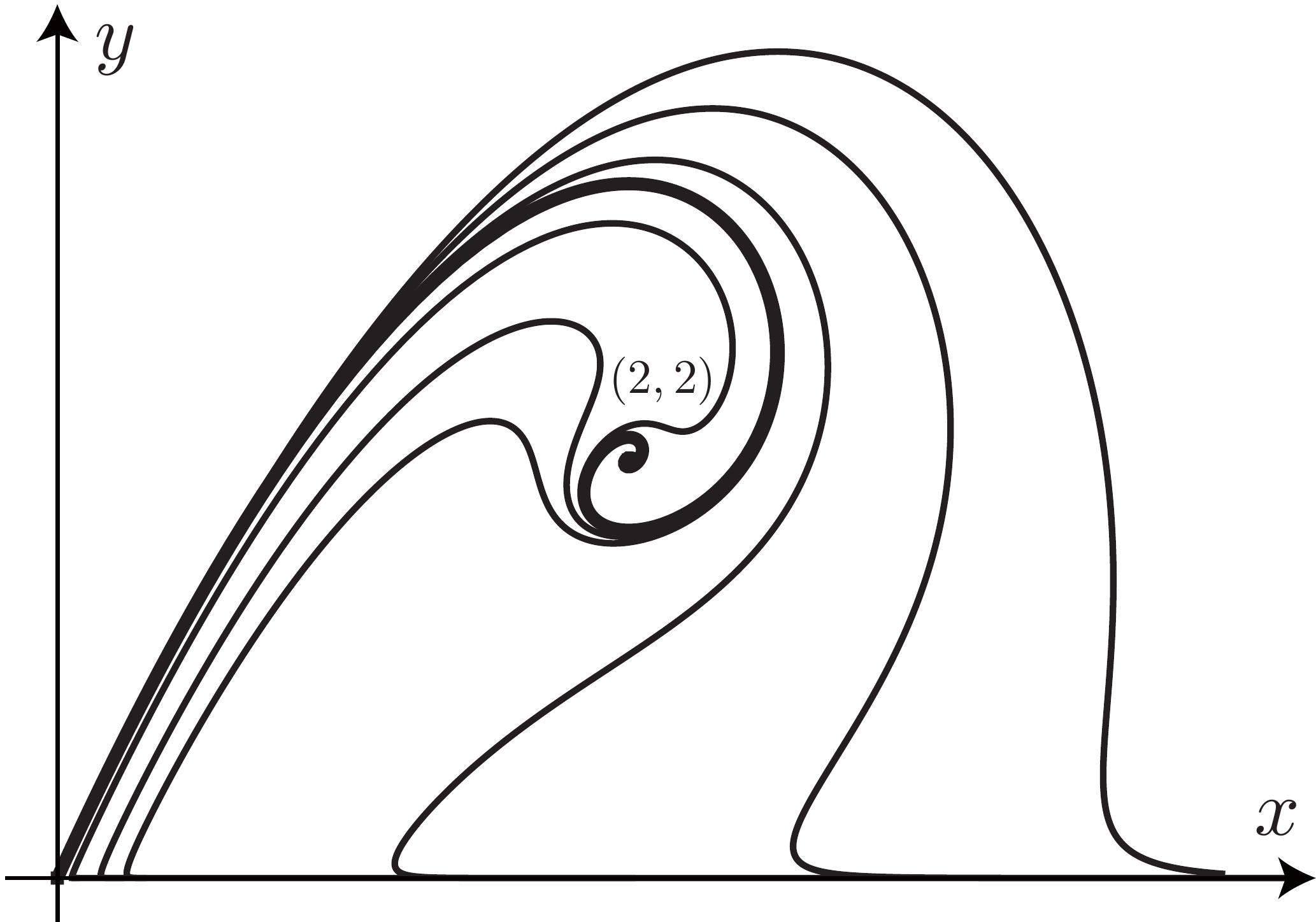}
\end{center}
\caption{\it\small Left: the heteroclinic orbit ($\eta=0$) joining the points $(0,0)$ and $(2,2)$ in the {\rm (MB)} case corresponds to the bold curve. Orbits corresponding to various positive values of $\eta$: $0.05$, $0.03$, $0.01$, $0.002$, $0.001$ and $0.0005$ in the {\rm (sFD)} case are approaching the one corresponding to the $\eta=0$ case, up to the singular point of the spiral.\label{Fig:mvs0}}
\end{figure}

\begin{figure}[ht]
\includegraphics[width=8cm]{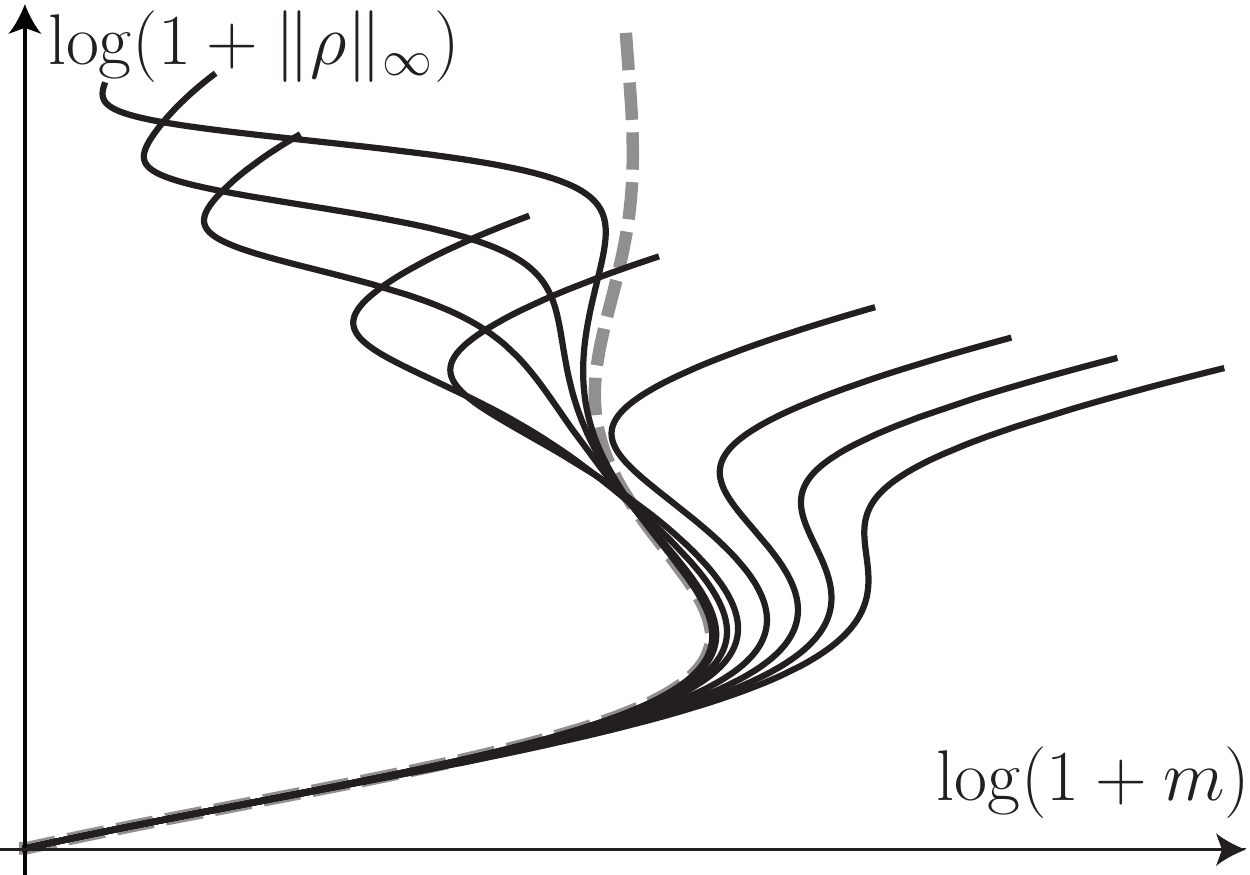}
\caption{\it\small Bifurcation diagrams of the solutions in the {\rm (sFD)} case, in the $(\log(1+m),\log(1+\|\rho\|_\infty))$ representation. Here $m=\frac M{4\pi}$is the normalized mass. Plain curves correspond to $\eta=0.05$, $0.04$, $0.03$, $0.02$, $0.01$, $0.002$, $0.001$, $0.0005$ while the dashed line corresponds to the limit~{\rm (MB)} case ($\eta=0$) and oscillates around a limiting value corresponding to $m=2$, \emph{i.e.}, $M=8\pi$.\label{Fig:mvs}}
\end{figure}

For the convenience of the reader and to avoid further repetitions, let us synthetize our notation:
\begin{enumerate}
\item[(MB)] \emph{Maxwell--Boltzmann statistics:}
\[
F(z)=e^z\;,\quad H(z)=\log z\quad\mbox{and}\quad R(z)=z\;,
\]
\item[(sFD)] \emph{simplified Fermi--Dirac statistics:}
\[
F(z)=H^{-1}(z)\;,\quad H(z)=\log z+\frac32\,\eta\,z^{2/3}\quad\mbox{and}\quad R(z)=\frac1{\frac1z+\frac\eta{z^{1/3}}}\;,
\]
\item[(FD)] \emph{Fermi--Dirac statistics:}
\[
F(z)=\frac\mu2\,f_{1/2}(z)\;,\quad H(z)=f_{1/2}^{-1}(2z/\mu)\quad\mbox{and}\quad R(z)=\frac\mu4\,f_{-1/2}\( f_{1/2}^{-1}(2z/\mu)\)\;,
\]
where $\mu$ is given in terms of $\eta$ by~\eqref{Eqn:EtaMu}, $f_{-1/2}$ has already been defined and $f_{-1/2}(z)=\int_0^\infty \frac{x^{-1/2}}{1+\exp(x-z)}dx$ is another Fermi function (\emph{cf.}~Appendix).
\end{enumerate}
Here the function $R$, which will be useful in the computations, is defined by $R=1/H'$, and we shall also consider $P$ such that $P'(z)=z\,H'(z)$. Unless specified otherwise, 
$\eta>0$ implicitly means that (FD) and (sFD) are under consideration, while statements corresponding to $\eta=0$ apply to (MB). However, let us emphasize that (MB) is a singular limit as $\eta\to0_+$ of the cases corresponding to $\eta>0$. It is the main purpose of this paper to clarify this issue. For simplicity, we shall omit to mention the dependence on $\eta$ and write $R=R_\eta$ only when emphasizing the dependence on~$\eta$.

As we shall see, part of the branches depicted in the bifurcation diagrams converges as $\eta\to0_+$ to the branches of (MB). For small values of the mass parameter both models share the same existence and uniqueness property of solution as was proved in \cite{DOLBEAULT:2009:HAL-00349574:2} by the generalized Rellich-Poho\-\v{z}aev method. However, there is a major difference in the asymptotic behavior when comparing the Fermi--Dirac and the Maxwell--Boltzmann cases. In the (MB) framework, there is a critical mass above which no stationary solution exists. In the (FD) and (sFD) cases, no such critical mass appears as was shown in~\cite{MR2136979}, and for any mass there exists a solution. Numerically, it is quite clear how the spiral of Fig.~\ref{Fig:Max} gets regularized in Fig.~\ref{Fig:mvs0}.

It is also rather clear that our results can be extended at almost no cost to a large class of nonlinearities $F$ depending on a parameter $\eta$, with appropriate properties, that converge to the exponential function in the limit as $\eta\to0_+$. From the physics point of view, however, it is the (FD) case that makes sense and this is why we have chosen not to cover the most general case but only the (FD) and (sFD) nonlinearities.

\medskip Before entering in the details, let us give a brief account of the literature on the subject. The reader is invited to refer to the references given in the quoted papers, especially for historical developments of the subject. In the three-dimensional (FD) case, it has been shown in \cite{MR2136979} that there exists a global branch of solutions of \eqref {Eqn:Fermi--Dirac} with arbitrary masses $M>0$. This result is achieved by a variational approach as in \cite{MR2143357,MR2136979,Wol} and by topological arguments (also see \cite{MR2506790}) based on the fixed point
\[\label{Eqn:FixedPoint}
\phi =\Delta^{-1}F(\lambda_\phi-\phi)\;,
\]
where $\lambda_\phi\le F^{-1}(M)$ is chosen in order to satisfy the mass constraint in~\eqref{Eqn:Fermi--Dirac}. Still in the framework of Fermi--Dirac statistics, see \cite{MR2092680,MR2338354,DOLBEAULT:2009:HAL-00349574:2,MR2548877,MR2506790} for stationary models and \cite{MR2675439,MR2305349,MR2295189} for the corresponding equations of evolution. In the (MB) case, the problem reduces to the classical Gelfand problem, \emph{cf.}~\cite{MR1305221,MR0340701}. A related family of problems can be considered when the $\rho$ factor in \eqref{Eqn:Evol} is replaced by $\rho\,P'(\rho)$, where the pressure function $P$ is given in terms of $H$ by $P'(\rho)=\rho\,H'(\rho)$. Such equations are motivated by~\cite{MR2092680,CSR} and have been mathematically studied in~\cite{MR2099972, MR2143357, biler2004parabolic, stanczy-evolution}.

\medskip This paper is organized as follows. In Section~\ref{Sec:Dynamical}, we will generalize the change of variables that has been done in the proof of Proposition~\ref{Prop:MB} for $\eta=0$ to the case $\eta\ge0$ and get a non-autonomous dynamical system. In the next section we shall focus our attention on some \emph{a priori} estimates for the dynamical system. Section~\ref{Sec:Mass} is devoted to the study of the dependence of the supremum norm of the density on the mass parameter. Finally, in Section~\ref{Sec:Continuity} we establish some continuity results and prove the existence of multiple solutions (Theorem~\ref{Thm:Main}) as suggested by the bifurcation diagrams. Detailed numerical results and plots of the bifurcation diagrams in the (sFD) case are presented in Section~\ref{Sec:Bifurcation}, which also contains plots of quantities of physical interest like the free energy. Technical results concerning Fermi and related functions can be found in Appendix~\ref{Appendix}.

\section{The dynamical system}\label{Sec:Dynamical}

As a preliminary observation, we recall that, according to the symmetry result of Gidas, Ni and Nirenberg in \cite[Theorem~1]{MR544879}, any solution~$\phi$ to~\eqref{Eqn:Fermi--Dirac} is radially symmetric when $F$ corresponds to (FD), (sFD) or (MB). In this section we shall prove that all radial solutions to \eqref{Eqn:Fermi--Dirac} can be parametrized by the solutions to some dynamical system. With a standard abuse of notation, we shall write a radial function of $x$ as a function of $|x|$. From $\rho=F\(\lambda-\phi\)$, \emph{i.e.}, $\lambda-\phi=H(\rho)$, we deduce that
\[
\phi'=-\,H'(\rho)\,\rho'\;.
\]
On the other hand, if we integrate the Poisson equation
\[
r^{-2}\,\(r^2\,\phi'\)'=\rho
\]
once, then by smoothness of $\phi$ we know that $\lim_{r\to 0^+} r^2\,\phi'(r)=0$ and thus get
\[
r^2\,\phi'=\zeta(r)
\]
where
\be{Eqn:zeta}
\zeta(r):=\frac1{4\pi}\int_{B(0,r)}\rho\;dx=\int_0^r z^2\,\rho(z)\;dz\;.
\ee
Altogether, we have obtained that
\[
\rho'+\frac1{H'(\rho)}\,\frac{\zeta(r)}{r^2}=0
\]
and, using $\rho=r^{-2}\,\zeta'(r)$ and $R(z):=1/H'(z)$, we find that
\be{Eqn:Poisson}
\(r^{-2}\,\zeta'\)'+r^{-2}\,R\(r^{-2}\,\zeta'\)\,\zeta=0\;.
\ee
Following the computations of \cite{BDEMN, Stanczy07}, we introduce the following change of variables
\be{Eqn:change}\left\{\begin{array}{l}
r=e^s\\[6pt]
x(s)=\zeta(r)/r\\[6pt]
y(s)=\zeta'(r)
\end{array}\right.\ee
and find by differentiation of $x(s)=e^{-s}\,\zeta(e^s)$ that
\[
x'=y-\,x
\]
while \eqref{Eqn:Poisson} can be rewritten as
\[
(e^{-2s}\,y)'\,e^{-s}+e^{-s}\,R\(e^{-2s}\,y\)\,x=0\;,
\]
thus showing that the dynamical system $s\mapsto(x(s),y(s))$ obeys the equations
\begin{eqnarray}
&&x'=y-\,x\;,\label{Eqn:DynamicalSystem-a}\\
&&y'=2\,y-\,e^{2s}\,R\(e^{-2s}\,y\)\,x\;.\label{Eqn:DynamicalSystem-b}
\end{eqnarray}

If $R(z)=z$, which corresponds to (MB), we recover the autonomous dynamical system \eqref{Syst:Auton} that was obtained in the proof of Proposition~\ref{Prop:MB}. For (FD) and (sFD), the main difficulty is due to the fact that \eqref{Eqn:DynamicalSystem-a}--\eqref{Eqn:DynamicalSystem-b} is not autonomous. However, our strategy is built on the observation that $R$ converges to the identity as $\eta\to0_+$.

We observe that \eqref{Eqn:DynamicalSystem-a}--\eqref{Eqn:DynamicalSystem-b} is related to \eqref{Eqn:Fermi--Dirac} by
\be{Eqn:Lim}
M=4\pi\,x(0)\quad\mbox{and}\quad\rho(0)=\|\rho\|_\infty=\lim_{s\to-\infty}e^{-2s}\,y(s)\;,
\ee
where by $\|\rho\|_\infty$ we denote the supremum norm $\|\rho\|_{L^\infty(B)}$. By the change of variables~\eqref{Eqn:change}, $0<r\le 1$ means that $-\infty<s\le 0$ and $(x,y)$ has to be confined in the \emph{positive quadrant} (upper right quadrant), \emph{i.e.}, $(x,y)\in[0,\infty)^2$. Altogether we can summarize our observations in the following result.
\begin{lemma}\label{Lem:Parametrization} The solutions $(x,y)$ of the dynamical system \eqref{Eqn:DynamicalSystem-a}--\eqref{Eqn:DynamicalSystem-b} defined on the interval $(-\infty,0]$ and satisfying conditions~\eqref{Eqn:Lim} are related to the solutions $\rho=F\(\lambda-\phi\)$ of \eqref{Eqn:Fermi--Dirac} by the changes of variables defined by \eqref{Eqn:zeta}~and~\eqref{Eqn:change}.\end{lemma}

\section{\emph{A priori} estimates on the dynamical system}\label{Sec:APriori}

In this section we establish some \emph{a priori} estimates on the solutions of \eqref{Eqn:DynamicalSystem-a}--\eqref{Eqn:DynamicalSystem-b}. We start with some  observations on invariant regions.
\begin{lemma}\label{Lem:APriori} Consider the dynamical system \eqref{Eqn:DynamicalSystem-a}--\eqref{Eqn:DynamicalSystem-b} with $R$ corresponding to one of the three cases, {\rm (MB)}, {\rm (FD)}, or {\rm (sFD)}, for some $\eta\ge0$. Then $R$ is continuous and such that $0\le R(z)\le z$ for any $z\in\R^+$. As a consequence, the $y=0$ axis is a stable manifold under the action of the flow, while the half-line $y=3\,x$, $x>0$ is tangent to the unstable manifold. Moreover, all trajectories with $y\ge 3\,x$, $x>0$ at $s=0$ are out of the positive quadrant for some negative time or, to be more specific, are such that $x(s)<0$ or $y(s)<0$ for any $s<0$ large enough.\end{lemma}
\begin{proof} The estimate $R(z)\le z$ in the (FD) case will be shown in Lemma~\ref{Lem:Equivalents}, in the Appendix. It is straightforward in the other cases.

The $y=0$, $x>0$ half-line is stable under the action of the flow because $R(0)=0$ and $(0,0)$ is an attraction point along this half-line. On the $x=0$, $y>0$ half-line the vector field points inwards the positive quadrant. Hence the positive quadrant is stable. The point $(0,0)$ is a stationary point with stable and unstable directions given respectively by $(1,0)$ and $(1,3)$.

In the positive quadrant, from the inequalities
\[
y'\le2\,y\quad\mbox{and}\quad3\,x'-\,y'\ge-\,(3\,x-\,y)\;,
\]
we deduce that
\[
y(s)\ge y(0)\,e^{2s}\quad\mbox{and}\quad3\,x(s)\,e^s\le3\,x(0)-\,y(0)+y(s)\,e^s
\]
for $s<0$, so that we reach a contradiction, namely $x(s)<0$ if $3\,x(0)-\,y(0)<0$ and $|s|$ is taken large enough. If $3\,x(0)-\,y(0)=0$, we get from
\be{Eqn:Asymp}
(3\,x-\,y)'=\,-\,(3\,x-\,y)+e^{2s}\,R\(e^{-2s}\,y\)\,x
\ee
that $(3\,x(t)-\,y(t))\,e^t<0$ for some $t<0$, arbitrarily small, and again reach a contradiction by the estimate
\be{Eqn:Asympx}
3\,x(s)\,e^s\le3\,x(t)\,e^t-\,y(t)\,(e^t-\,e^s)\;,\quad s<t
\ee
if we assume that $x(s)$ and $y(s)$ are positive for any $s<t$. It should be noted that the crucial monotonicity of $(3\,x(t)-y(t))\,e^t$ for $x(t)>0$ was used in the argument above to get the conclusion. \end{proof}

A straightforward consequence of Lemma~\ref{Lem:APriori} is that $(0,0)$ is the unique stationary point in \hbox{$\{(x,y)\in\R^2\,:\,x\ge0\;\mbox{and}\;y\ge0\}$}. Hence we know that
\[
\lim_{s\to-\infty}(x(s),y(s))=(0,0)\;.
\]
Notice that this is compatible with \eqref{Eqn:Lim}. A slightly more precise asymptotic description of the solutions goes as follows.
\begin{lemma}\label{Lem:Limits} Any solution of \eqref{Eqn:DynamicalSystem-a}--\eqref{Eqn:DynamicalSystem-b} satisfying \eqref{Eqn:Lim} and confined to the positive quadrant is such that
\[
\lim_{s\to-\infty}e^{-2s}\,y(s)=3\lim_{s\to-\infty}e^{-2s}\,x(s)=\rho(0)\;.
\]\end{lemma}
\begin{proof} The first limit $\lim_{s\to-\infty}e^{-2s}\,y(s)=\rho(0)$ arises from \eqref{Eqn:Lim}. From \eqref{Eqn:Asymp} and \eqref{Eqn:Asympx}, we know that
\[
\frac d{ds}\,(3\,x-\,y)\,e^s=e^{3s}\,R\(e^{-2s}\,y\)\,x\sim e^{3s}\,R\(\rho(0)\)\,x\quad\mbox{as}\quad s\to-\infty\;,
\]
so that $\lim_{s\to-\infty}e^{-2s}\,x(s)$ exists. By l'H\^opital's rule, we can compute
\[
L:=\lim_{s\to-\infty} \frac{x(s)}{y(s)}=\lim_{s\to-\infty} \frac{x'(s)}{y'(s)}=\lim_{s\to-\infty} \frac{y(s)-\,x(s)}{2y(s)-\,x(s)\,e^{2s}\,R(y(s)\,e^{-2s})}=\frac{1-L}2\;,
\]
whence $L=1/3$ follows.\end{proof}

Notice that the existence of a branch of solutions follows from the existence theory for elliptic equations that can be found, among others, in \cite{MR2136979} and \cite{MR2506790}.

\section{\emph{A priori} estimates depending on the mass}\label{Sec:Mass}

Let us notice that
\be{Eqn:MassEstSimple}
M\le\frac{4\pi}3\,\|\rho\|_\infty
\ee
readily follows from the definition of $M=\int_B \rho\,dx$. Here $B$ is the unit ball in $\R^3$.
\begin{lemma}\label{Lem:MassSupNorm} Consider a solution of \eqref{Eqn:Fermi--Dirac} with mass $M$ and let $m:=\frac M{4\pi}$. With $F=H^{-1}$ and $R=1/H'$ corresponding to one of the three cases {\rm (MB)}, {\rm (FD)} or {\rm (sFD)}, we have the estimate
\be{Eqn:MassEstDifficult}
2\,H(\|\rho\|_\infty)-R(\|\rho\|_\infty)\,m\le m+\,2\,H(3\,m)\;.
\ee
\end{lemma}
\begin{proof} Indeed, from \eqref{Eqn:DynamicalSystem-b} and $H'=1/R$ we get the following equality
\[
\frac d{ds}\big(H(y(s)\,e^{-2s}\big)=\frac{\big(y(s)\,e^{-2s}\big)'}{R\big(y(s)\,e^{-2s}\big)}=-\,x(s)\;.
\]
By integrating this identity on the interval $(-\infty,0]$ and using \eqref{Eqn:Lim}, one obtains
\[\label{Eqn:Tes}
H(y(0))-\,H(\|\rho\|_\infty)=-\int_{-\infty}^0 x(s)\;ds\;.
\]
Then by integrating \eqref{Eqn:DynamicalSystem-a} and using $\lim_{t\to-\infty}x(t)=0$, according to Lemma \ref{Lem:Limits}, we get $x(0)=\int_{-\infty}^0 x'(s)\;ds=\int_{-\infty}^0 y(s)\;ds-\int_{-\infty}^0 x(s)\;ds$. Hence one arrives at
\be{Eqn:Teq}
H(y(0))-\,H(\|\rho\|_\infty)=-\int_{-\infty}^0 y(s)\;ds+x(0)\;.
\ee
Next we may integrate \eqref{Eqn:DynamicalSystem-b} and use $\lim_{s\to-\infty}y(s)=0$, according to Lemma \ref{Lem:Limits}, and the monotonicity of $R$ (see the expression of the Fermi function $f_\alpha$ in Appendix~\ref{Appendix} in the (sFD) case), to get
\begin{multline*}
y(0)=2\int_{-\infty}^0 y(s)\;ds-\int_{-\infty}^0 e^{2s}\,R\(e^{-2s}\,y(s)\)\,x(s)\;ds\\
\ge2\int_{-\infty}^0 y(s)\;ds-\int_{-\infty}^0 e^{s}\,R(\|\rho\|_\infty)\,x(0)\;ds
\end{multline*}
where the inequality follows from $\lim_{s\to-\infty}y(s)\,e^{-2s}=\|\rho\|_\infty$ and $\(y(s)\,e^{-2s}\)'\le0$, as shown in Section~\ref{Sec:APriori}. Inserting this estimate into \eqref{Eqn:Teq} we finally have
\[
H(y(0))-\,H(\|\rho\|_\infty)\ge x(0)-\frac12\,y(0)-\frac12\,x(0)\,R(\|\rho\|_\infty)\;,
\]
which is exactly the claim of the lemma because of \eqref{Eqn:Lim} together with $R\ge 0$ implying monotonicity of $H$  and Lemma~\ref{Lem:Limits} with its proof implying $3\,m=3\,x(0)\ge y(0)$.\end{proof}

Recall that solutions exist for arbitrary large masses in the (FD) and (sFD) cases, that is for $\eta>0$, according to~\cite{MR2136979}. Estimates~\eqref{Eqn:MassEstSimple} and \eqref{Eqn:MassEstDifficult} provide an interesting qualitative property of the branch of solutions for large values of the mass, which is of interest by itself and provides some insight on numerical results of Section~\ref{Sec:Bifurcation}. The proof of the following corollary is straightforward by \eqref{Eqn:MassEstSimple} and \eqref{Eqn:MassEstDifficult}.
\begin{corollary}\label{Cor:LargeMasses} Consider the solutions to~\eqref{Eqn:Fermi--Dirac}, parametrized by $M$, in the {\rm (FD)} and {\rm (sFD)} cases, for $\eta>0$. Then $\|\rho\|_\infty\nearrow\infty$ if and only if $M\nearrow\infty$. \end{corollary}

\section{Main continuity results}\label{Sec:Continuity}

In this section we establish the convergence as \hbox{$\eta\to0_+$} of the Fermi--Dirac (FD) model and the simplified Fermi--Dirac (sFD) model to the Maxwell--Boltzmann (MB) model. The main difficulty is that we deal with a non-compact interval and exponential factors.

Let us consider solutions to \eqref{Eqn:DynamicalSystem-a}--\eqref{Eqn:DynamicalSystem-b} and define $S_\eta(z):=z-R_\eta(z)$. An elementary but useful estimate is established in Lemma~\ref{Lem:Equivalents}, in the Appendix, which shows that for some positive positive constant $C$ which depends neither on $\eta$ nor on $\mu$, we have
\be{Eqn:Set}
0\le S_\eta(z)\le C\,\eta\,z^{5/3}\quad\forall\,z\in(0,\infty)\;.
\ee
In order to emphasize the dependence on $\eta\ge0$ we shall denote the solution by $(x_\eta,y_\eta)$.
\begin{lemma}\label{Lem:Continuity} Let $\rho_0>0$ and take any $\rho\le \rho_0$. With the above notations, for any $\eta>0$, consider solutions to \eqref{Eqn:DynamicalSystem-a}--\eqref{Eqn:DynamicalSystem-b} in either the (sFD) or in the (FD) case such that
\be{Eqn:Limits}
\lim_{s\to -\infty} e^{-2s}\,y_\eta(s)=3\,\lim_{s\to -\infty} e^{-2s}\,x_\eta(s)=\rho\;.
\ee
Then as $\eta\to0_+$, $(x_\eta,y_\eta)$ uniformly converges on $(-\infty,0]$ to the solution $(x_0,y_0)$ of~\eqref{Syst:Auton} satisfying \eqref{Eqn:Limits} with $\eta=0$ and, also, $(x_\eta',y_\eta')$ 
uniformly converges to $(x_0',y_0')$.\end{lemma}
\begin{proof}Let us start with some preliminary estimates. As in Section~\ref{Sec:APriori}, we know that
\begin{eqnarray*}
&&\(y_\eta\,e^{-2s}\)'=-\,e^{2s}\,R_\eta\(e^{-2s}\,y_\eta\)\,x_\eta\le0\;,\\
&&\(e^s\,(3\,x_\eta-\,y_\eta)\)'=R_\eta\(e^{-2s}\,y_\eta\)\,e^{3s}\,x_\eta\ge0\;,\\
&&\(x_\eta\,e^{-2s}\)'=-\,e^{-2s}\,(3\,x_\eta-\,y_\eta)\le0\;.
\end{eqnarray*}
Taking into account \eqref{Eqn:Limits}, we deduce that
\begin{eqnarray*}
&&y_\eta(s)\le\rho_0\, e^{2s}\;,\\
&&3\,x_\eta(s)\ge y_\eta(s)\;,\\
&&x_\eta(s)\le\tfrac13\,\rho_0\, e^{2s}\;.
\end{eqnarray*}
for any $s\le0$. These estimates are uniform with respect to $\eta\ge0$. Combined with~\eqref{Eqn:Set}, they also imply that
\[\label{Eqn:Sess}
\eta\,f_\eta(s)\,e^{4s}:=S_\eta\(e^{-2s}\,y(s)\)\,e^{2s}\,x\le\tfrac13\,C\,\eta\,\rho_0^{8/3}\,e^{4s}\quad\forall\;s\le0
\]
for some uniformly bounded functions $f_\eta$, with $\eta\ge0$ small: $f_\eta(s)\le\tfrac13\,C\,\rho_0^{8/3}$ for any $s\le0$. Moreover, the difference of $(x_\eta,y_\eta)$ and $(x_0,y_0)$ satisfies the system
\[\label{Eqn:DiffSystem}
\left\{\begin{array}{l}
(x_\eta-x_0)'=(y_\eta-y_0)-(x_\eta-x_0)\;,\\[6pt]
(y_\eta-y_0)'=2\,(y_\eta-y_0)-\(y_\eta\,(x_\eta-x_0)+x_0\,(y_\eta-y_0)\)+\eta\,f_\eta(s)\,e^{4s}
\end{array}\right.
\]
with $f_\eta(s)\le\kappa$. Let us define
\[
A_\eta(t):={\rm sup}_{s\le t}\,e^{-2s}\,|x_\eta(s)-x_0(s)|\quad\mbox{and}\quad B_\eta(t):={\rm sup}_{s\le t}\,e^{-2s}\,|y_\eta(s)-y_0(s)|\;.
\]
We know that $\lim_{t\to-\infty}A_\eta(t)=\lim_{t\to-\infty}B(t)=0$. From the above system we deduce that the estimates
\begin{eqnarray*}
&&\frac d{dt}\(e^{3t}\,A_\eta\)\le e^{3t}\,B_\eta\;,\\
&&\frac d{dt}\,B_\eta\le\(\rho_0\,A_\eta+\tfrac13\,\rho_0\,B_\eta+\eta\, \kappa\)e^{2t}\;,
\end{eqnarray*}
hold for almost all $t<0$. An integration on $(-\infty,t)$ shows that
\begin{eqnarray*}
&&A_\eta\le\tfrac13\,B_\eta\;,\\
&&B_\eta\le\frac12\(\rho_0\,A_\eta+\tfrac13\,\rho_0\,B_\eta+\eta\ \kappa \)e^{2t}\;,
\end{eqnarray*}
for any $t<0$. Altogether, using a Gronwall estimate, we finally arrive at
\[
0\le 3\,A_\eta(t)\le B_\eta(t)\le\frac 12\,\eta\, \kappa \,e^{\rho_0/3}\quad\forall\,t\in(-\infty,0]\;,
\]
which concludes the proof.
\end{proof}

\begin{proof}[Proof of Theorem~\ref{Thm:Main}] Let $\varepsilon>0$ and consider the enlarged singular set of masses
\[
\mathcal S_\varepsilon:=\cup_{n=1}^{\infty} (M_*^n-\varepsilon,M_*^n+\varepsilon)\cup_{n=1}^{\infty} (M^*_n-\varepsilon,M^*_n+\varepsilon)\;.
\]
If $\varepsilon$ is small enough, then $\mathcal M_\varepsilon:=(0,M_1^*)\setminus \mathcal S_\varepsilon$ is non-empty. For any $M\in\mathcal M_\varepsilon$, the map $\rho_0\mapsto x_0(0)=\frac M{4\pi}$ is smooth and locally invertible, if $(x_0,y_0)$ denotes the solution to \eqref{Syst:Auton} satisfying \eqref{Eqn:Limits}. This property is elementary (see the proof of Proposition~\ref{Prop:MB} for details), and it is is also satisfied by the solution $(x_\eta,y_\eta)$ to~\eqref{Eqn:DynamicalSystem-a}--\eqref{Eqn:DynamicalSystem-b} satisfying \eqref{Eqn:Limits}, at least for $\eta>0$ small enough. In fact for any $n\in \mathbb{N}$ and $M\in  (M_*^n,M_*^{n+1})\cup (M^*_n,M^*_{n+1})$ we can choose $\varepsilon>0$ small enough to have $M\in \mathcal M_\varepsilon$ and $\rho_0>0$ large enough to have, due to Lemma \ref{Lem:Continuity}, the uniform convergence of the branch of solutions, in variables $(M,\rho)$, corresponding to (FD) or (sFD) to (MB), as $\eta\rightarrow 0^+$ for any $\rho\le \rho_0$ in $\mathcal M_\varepsilon\times [0,\rho_0]$. There exists therefore a solution to~\eqref{Eqn:DynamicalSystem-a}--\eqref{Eqn:DynamicalSystem-b} with mass $M$ in a neighborhood of $(x_0,y_0)$ both in the (sFD) case and in the (FD) case. On $\mathcal M_\varepsilon$, the solutions in the (MB) case are isolated and in finite number, which allows to conclude the proof. It must be underlined that the similar conclusion holds for (FD) and (sFD) cases as for (MB) one but only for sufficiently small values of $\eta>0$. \end{proof}

\section{Bifurcation diagrams and numerical results}\label{Sec:Bifurcation}

This section is devoted to numerical computations of all solutions in the (sFD) case and goes beyond a simple illustration of Theorem~\ref{Thm:Main}. In particular we compute the solutions for all masses and also plot various quantities of physical interest
like entropy or free energy. Recall that solutions can be characterized as critical points of the free energy under mass constraint.

For numerical computations, it is convenient to introduce yet another change of variables
\be{Eqn:Change2}
p(s):=e^{-2s}\,y(s)\quad\mbox{and}\quad q(s):=e^{-2s}\,x(s)
\ee
which solve
\[
\left\{\begin{array}{l}
q'=p-3\,q\;,\\[6pt]
p'=\,-\,R\(p\)\,e^{2s}\,q\;,
\end{array}\right.
\]
supplemented with the conditions
\[
4\pi\,q(0)=M\;,\quad\lim_{s\to-\infty}q(s)=\tfrac 13\,p_\infty\quad\mbox{and}\quad\lim_{s\to-\infty}p(s)=p_\infty\;.
\]
Our numerical scheme goes as follows. We fix an $\varepsilon>0$, small, and use the fact that $y(s)$ is of the order of $\varepsilon$ for $s=t(\varepsilon)$ if $e^{-2t(\varepsilon)}\,\varepsilon\approx p_\infty$, which determines the dependence of $t(\varepsilon)$ on $\varepsilon$. Hence our approximated solution is given by
\[
t(\varepsilon)=\frac12\,\log\(\frac{\varepsilon}{p_\infty}\)\,,\quad q(t(\varepsilon))=\tfrac 13\,p_\infty\quad\mbox{and}\quad p(t(\varepsilon))=p_\infty\;.
\]
To emphasize the dependence on $\varepsilon$, we shall denote it by $(p^\varepsilon,q^\varepsilon)$. Figs.~\ref{F4} shows the trajectories in the $(p,q)$ phase space.

\begin{figure}[ht]
\includegraphics[width=4cm]{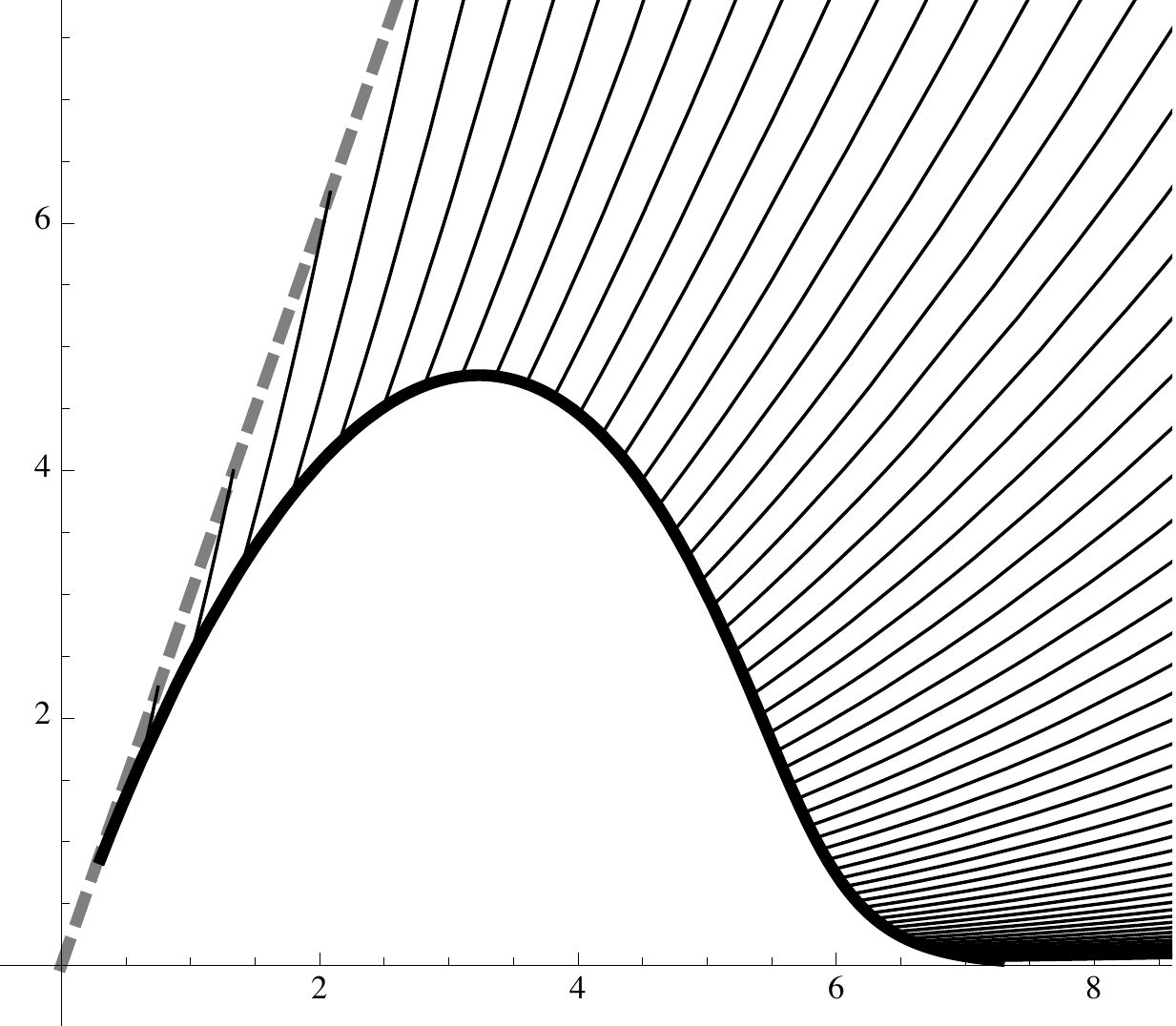} \hspace*{24pt}
\includegraphics[width=4cm]{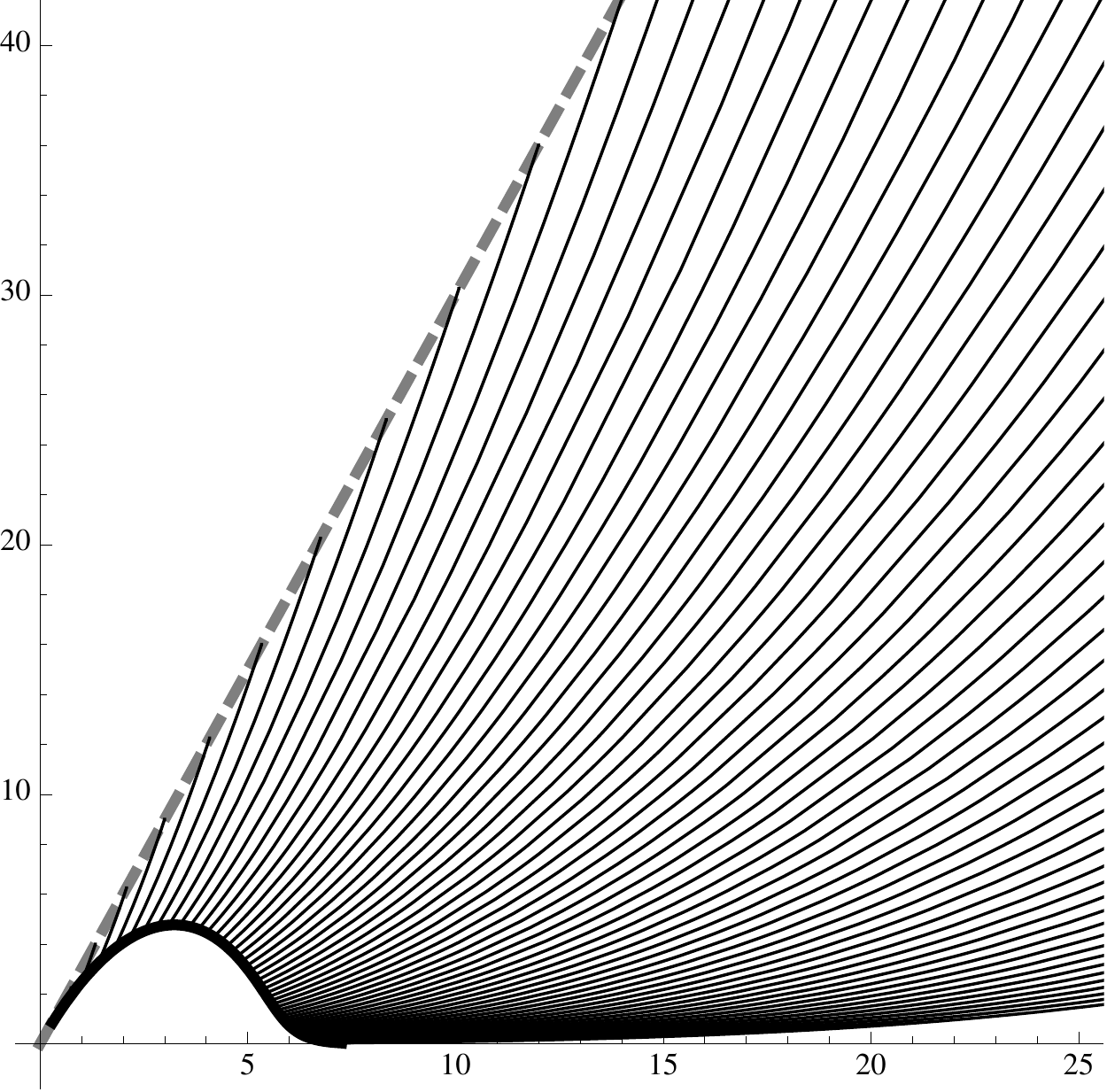}
\caption{\it\small Right: approximating solutions connect the $p=3\,q$, $q>0$ (dashed) half-line corresponding to the initial datum $p=p_\infty$ at $s=t(\varepsilon)<0$, with~$|t(\varepsilon)|$ large, to the curve $q^\varepsilon\mapsto p^\varepsilon$ at time $s=0$. Left: enlargement. Both cases correspond to $\eta=0.1$ in the {\rm (sFD)} model. \label{F4}}
\end{figure}

With this representation, we can draw a bifurcation diagram, see Fig.~\ref{Fig:svm}, which covers a larger range of values of $M$ and $\|\phi\|_\infty$ than the ones taken into account in Fig.~\ref{Fig:mvs}.

\begin{figure}[ht]
\begin{center}
\includegraphics[width=8cm]{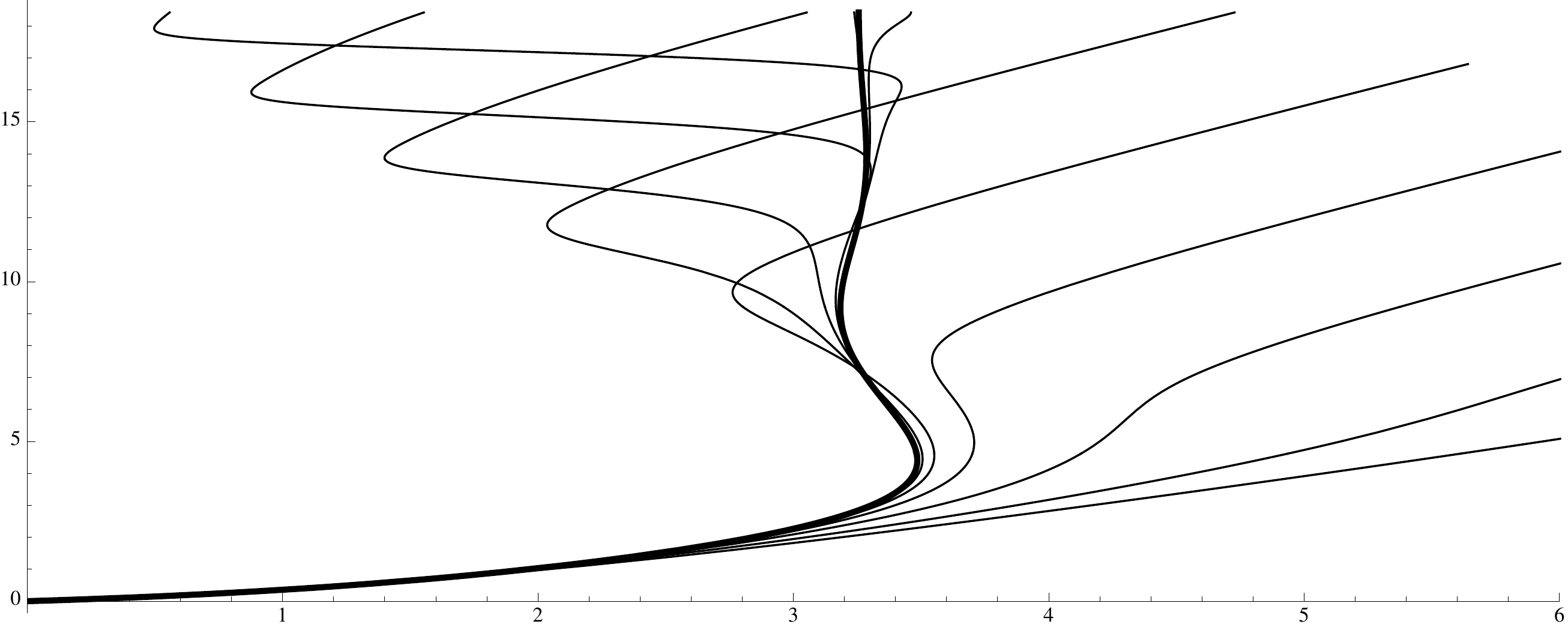}
\end{center}
\caption{\it\small Plot of $\log(1+\|\rho\|_\infty)$ in terms of $\log(1+M)$ for various values of $\eta$: $\eta=10^k$ with $k=-5$, $-4.5$,... $-0.5$,~$0$ in case of {\rm (sFD)}. The bold plain curve corresponds to the {\rm (MB)} case $\eta=0$.\label{Fig:svm}}
\end{figure}

The bifurcation diagrams show that as the parameter $\eta$ approaches zero, the solutions of simplified Fermi--Dirac (sFD) model parametrized by this number $\eta$ approach the ones for the Maxwell--Boltzmann (MB) case, corresponding to the limit value $\eta=0$, when the mass is in the admissible range for (MB). They also show that solutions with arbitrarily large masses exist as long as $\eta$ is positive.

In practice, for numerical purposes, we have chosen $\varepsilon=10^{-6}\ll p_\infty$. We may notice that our solutions come very close to $(0,0)$ for some $s<t(\varepsilon)$ but are such that $x(s)$ is negative for larger, negative values of $s$. The Gelfand problem is recovered for $\eta=0$ (see for instance \cite{DOLBEAULT:2009:HAL-00349574:2, MR0340701}), but Figure~\ref{Fig:svm} clearly shows that a solution exists for any $M>0$ as soon as $\eta$ is positive. This corresponds well to the known existence results which either by variational or topological arguments yield the solution irrespectively of any value of the mass parameter. For the detailed analysis of this issue, see \cite{MR2136979}.

\medskip Solutions to \eqref{Eqn:Fermi--Dirac} can be obtained by a fixed point method as in \cite{MR2136979}. Alternatively, they can be characterized variationally. Hence, it should be noted that to get the minimal solution of \eqref{Eqn:Fermi--Dirac}, one can minimize the \emph{free energy} functional of the solution (\emph{cf.}~\cite{MR2136979}), namely
\[
\mathcal F[\rho]:=\mathcal S[\rho]-\mathcal P[\rho]
\]
where $\mathcal S[\rho]:=\int_B\beta(\rho)\,dx$ is the \emph{generalized entropy} and $\mathcal P[\rho]:=\frac12\int_B|\nabla\phi|^2\,dx$ is the \emph{self-consistent potential energy}. The function $\beta$ is convex and such that $\beta'=H$, which means $\beta''=1/R$. Any solution to \eqref{Eqn:Fermi--Dirac} is a critical point of $\mathcal F$ under mass constraint. Equation~\eqref{Eqn:Evol} is, at least formally, the gradient flow of $\mathcal F$ with respect to Wasserstein's distance according to the ideas introduced, \emph{e.g.,} in~\cite{MR1842429} and the reader is invited to check that $\mathcal F$ is monotone non-increasing along the flow defined by~\eqref{Eqn:Evol}. This flows preserves the mass. Hence a minimizer of $\mathcal F$ under mass constraint is non-linearly dynamically stable.

\emph{Entropies} for the isothermal model and for a model with fixed energy have been exhibited for instance in \cite{biler2004parabolic}, in connection with many other papers dealing with \emph{generalized entropy} (or free energy) functionals, see for instance \cite{MR2065020, biler2004parabolic, MR1853037}. Note that for the isothermal case, the entropies in \cite{MR2143357} or \cite{biler2004parabolic} differ by the mass constant with the above one (which has no physical consequences since mass is conserved by the flow).  

In the general case, we may notice that the entropy generating function $\beta$ can be written as
\[
\beta(z)=z\,H(z)-P(z)\;,
\]
where the pressure $P$ is such that $P'(z)=z\,H'(z)$, while for the specific $R$ corresponding to (sFD) case (see Appendix~\ref{Appendix} for details) we have
\[
\beta(z)=z\,\log z-\,z+\frac9{10}\,\eta\,z^{5/3}\;.
\]
The generalized entropy $\mathcal S[\rho]$ can be computed as
\begin{multline*}
s(t):=4\,\pi\int_{-\infty}^0e^{3\,s}\,\beta\(e^{-2\,s}\,y(s)\)\;ds\\
=4\,\pi\int_{-\infty}^0e^{3\,s}\,\beta\(p(s)\)\;ds\approx4\,\pi\int_{t(\varepsilon)}^0e^{3\,s}\,\beta\(p^\varepsilon(s)\)\;ds
\end{multline*}
using \eqref{Eqn:zeta}, \eqref{Eqn:change} and \eqref{Eqn:Change2}. Similarly, the potential energy $\mathcal P[\rho]$ can be computed as
\[
p(t):=4\,\pi\int_{-\infty}^0e^s\,|x(s)|^2\,ds\approx4\,\pi\int_{t(\varepsilon)}^0e^{5\,s}\,|q^\varepsilon(s)|^2\,ds\;.
\]
\begin{figure}[hb]\begin{center}
\includegraphics[height=4cm]{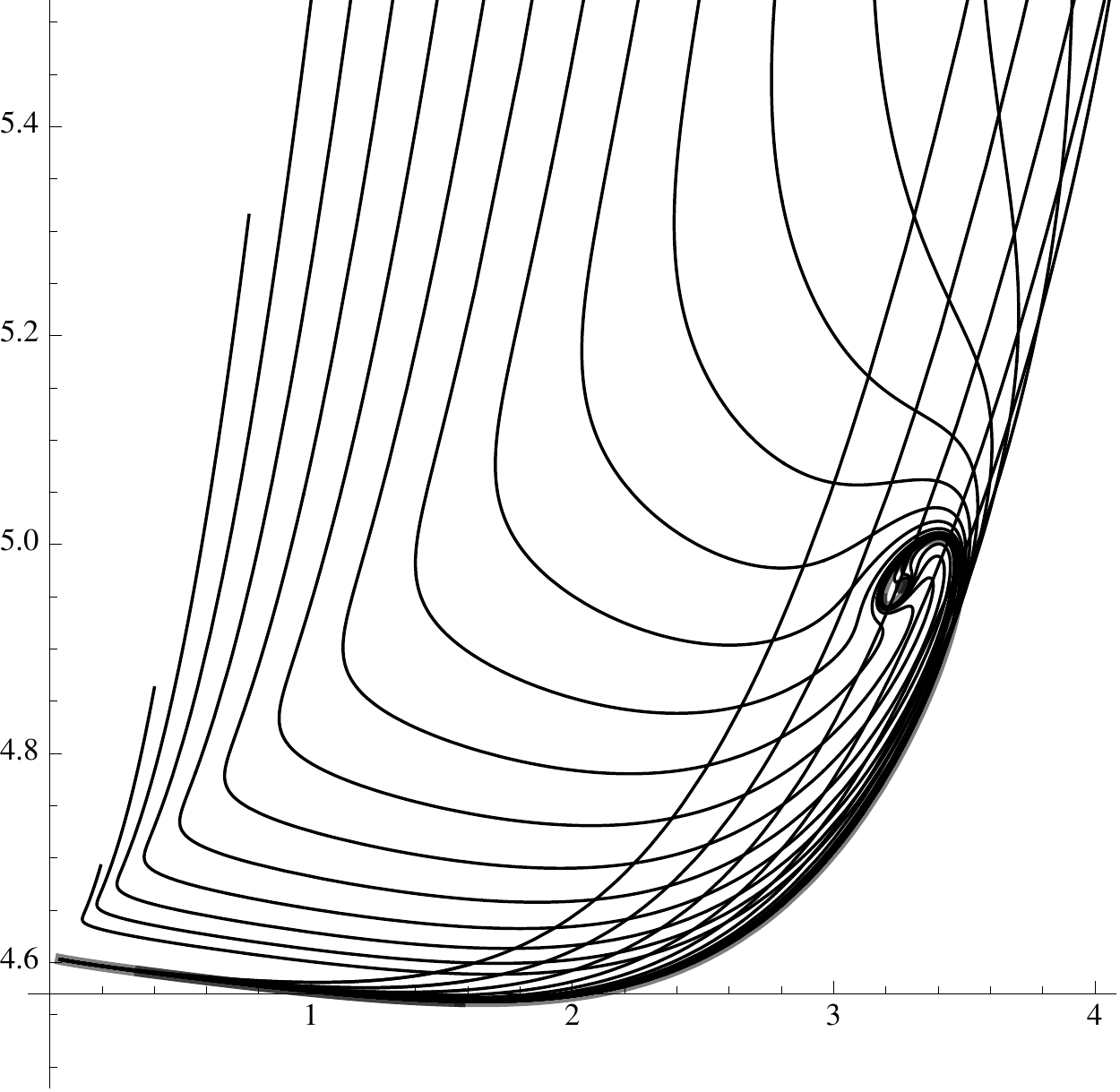} \hspace*{30pt}
\includegraphics[height=4cm]{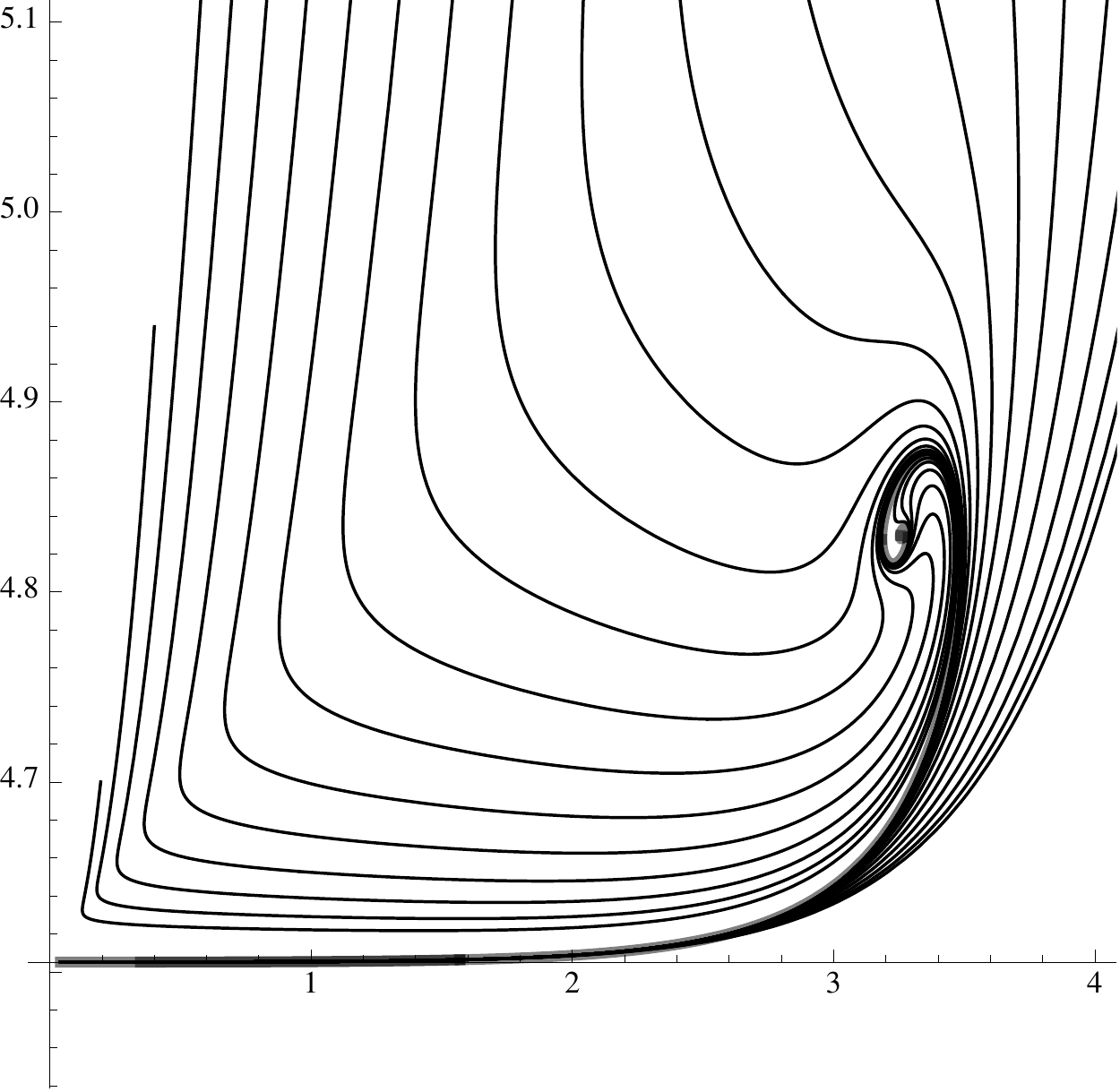}
\end{center}
\caption{\it\small Left: entropy as a function of the mass: plot of the curves $(\log(1+M),\log(100+\mathcal S[\rho]))$ for $\eta=10^k$ with $k=-5$, $-4.75$,... $-0.25$,~$0$. Right: potential energy as a function of the mass: plot of $(\log(1+M),\log(100+\mathcal P[\rho]))$ for $\eta=10^k$ with $k=-5$,$-4.75$,...,$-0.25$,$0$. In both cases, the bold curve corresponds to $\eta=0$. \label{Fig7}}
\end{figure}
Numerical results are shown in Figs.~\ref{Fig7} and~\ref{Fig8}.

\begin{figure}[ht]\begin{center}
\includegraphics[height=4cm]{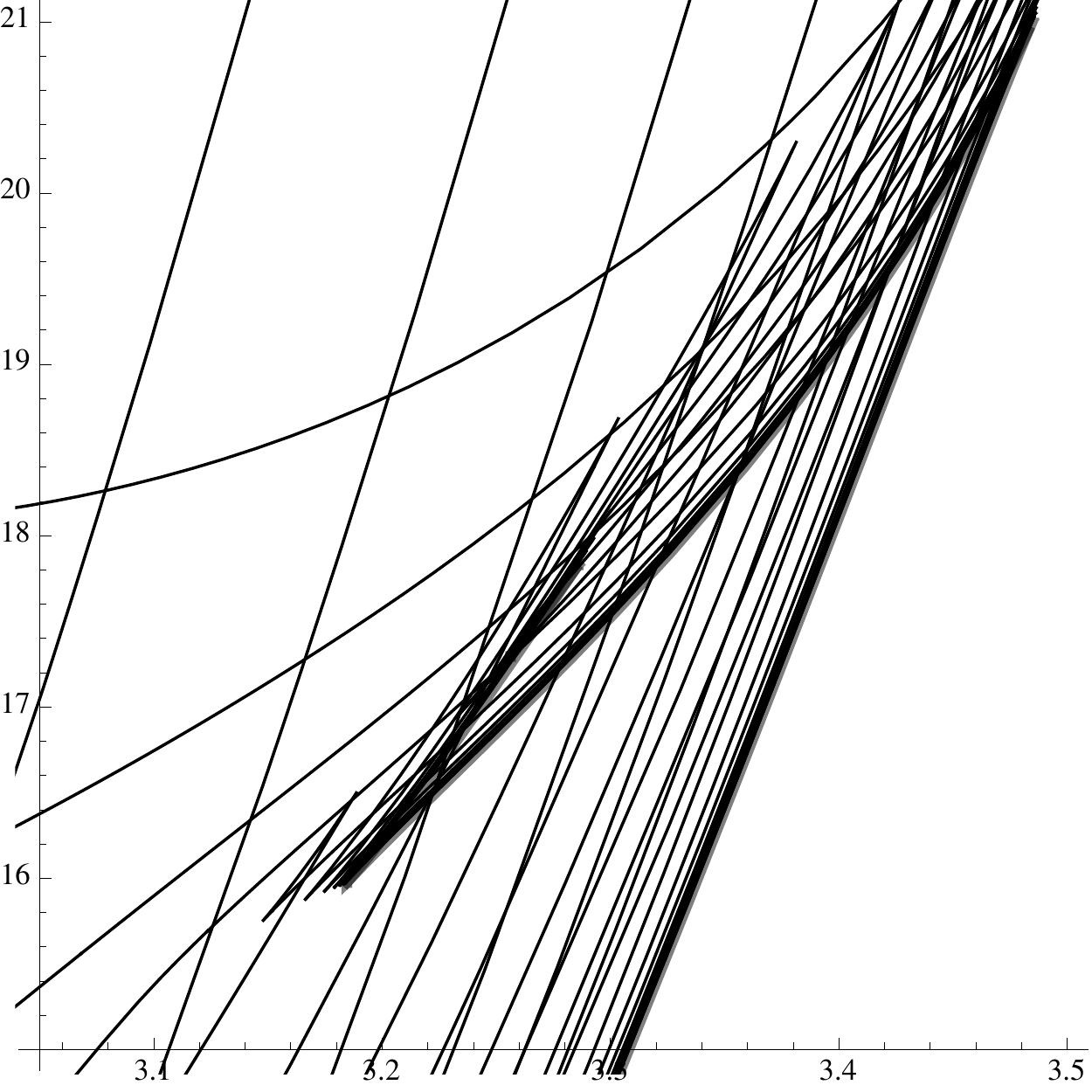} \hspace*{30pt}
\includegraphics[height=4cm]{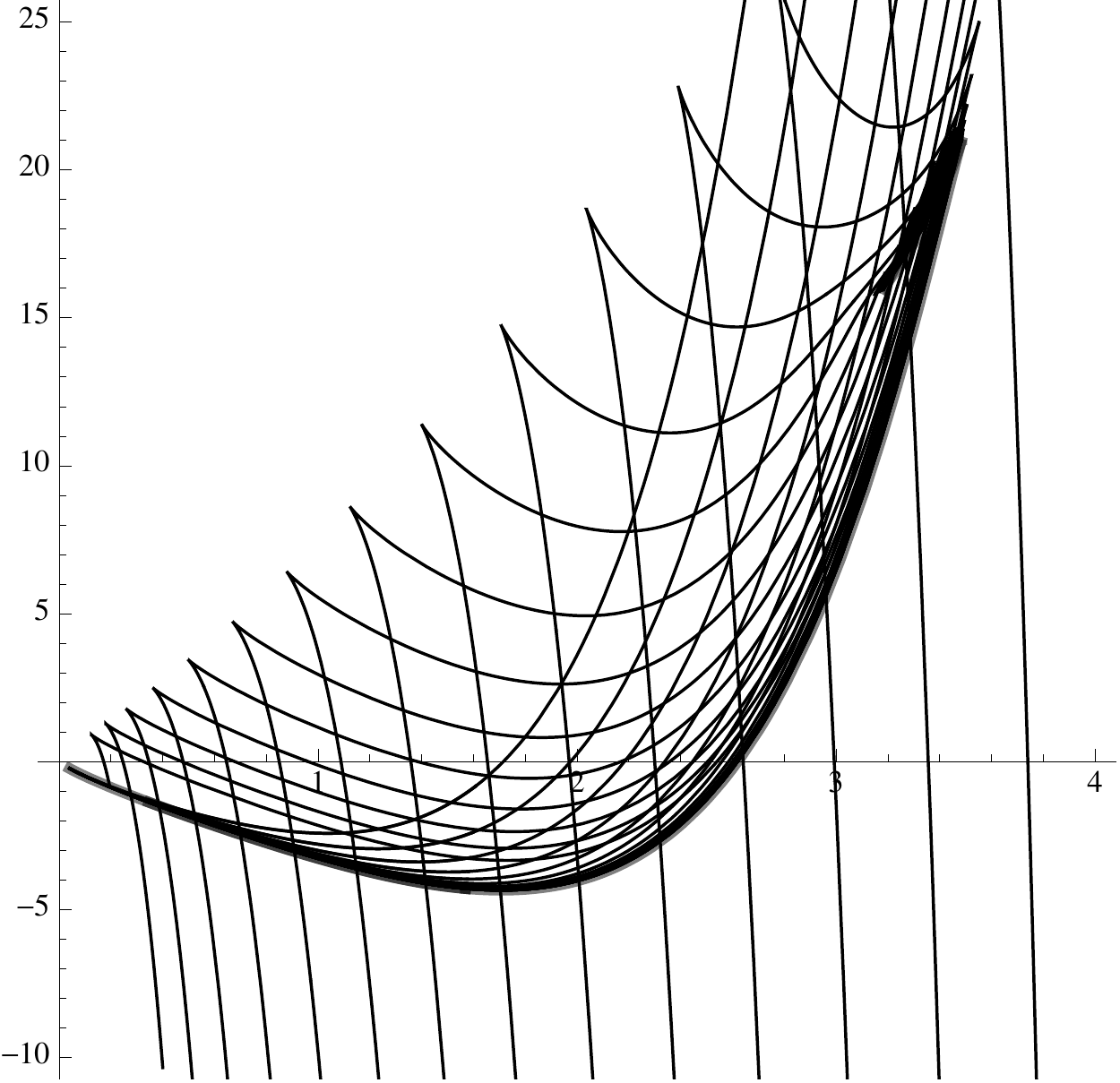}
\end{center}
\caption{\it\small The free energy as a function of the mass: plot of $(\log(1+M),\mathcal F[\rho]$ for $\eta=10^k$ with $k=-5$, $-4.75$,... $-0.25$,~$0$ (right).  The bold curve corresponds to $\eta=0$. An enlargement near the endpoint of the curve corresponding to $\eta=0$ is shown on the left.\label{Fig8}}
\end{figure}

\newpage\appendix\section{Fermi and related functions}\label{Appendix}

The Fermi function $f_\alpha$ is defined for any $\alpha>-1$ and $z\in \mathbb R$ as
\[\label{Eqn:Fermi-Function}
f_\alpha(z)=\int_0^\infty\frac{x^\alpha}{1+\exp\,(x-z)}\;dx\;.
\]
We refer to \cite{stanczy-evolution} for more details. In this paper we just use three values of~$\alpha$: $3/2$, $1/2$ and $-1/2$.
The Fermi functions enjoy the following recurrence identity
\[
f'_{\alpha+1}(z)=(\alpha+1)\,f'_\alpha(z)\quad\forall\,z\in\R
\]
and asymptotically behave like
\[\label{Eqn:FI}
f_\alpha(z)\sim\frac1{\alpha+1}\,z^{\alpha+1}\;\mbox{as}\;z\to+\infty\quad\mbox{and}\quad f_\alpha(z)\sim\Gamma(\alpha+1)\,e^z \;\mbox{as}\;z\to-\infty\;,
\]
where $\Gamma$ is the Euler Gamma function. The function $F$ in the Fermi--Dirac (FD) case is defined as
\[
F(z)=\frac\mu2\,f_{1/2}(z)\quad\forall\,z\in\R\;.
\]
{}From the relations $F=H^{-1}$, $R\,H'=1$ and $P'(z)=z\,H'(z)$, it follows that
\[\label{Eqn:RFD}
R(z)=\frac\mu4\,f_{-1/2}\( f_{1/2}^{-1}(2z/\mu)\)\quad\forall\,z\in(0,\infty)
\]
and, as a consequence of the relation $P'(z)\,R(z)=z$, we obtain
\[
P(z)=\frac\mu3\,f_{3/2}\( f_{1/2}^{-1}(2z/\mu)\)\quad\forall\,z\in(0,\infty)
\]
while the $H$ function is given by
\[\label{Eqn:H}
H(z)=f_{1/2}^{-1}(2z/\mu)\quad\forall\,z\in(0,\infty)\;.
\]
Note that to justify the above relations one can use the identity
\[
P'(z)=\frac2\mu \( f_{1/2}^{-1}\)'\(\frac{2z}{\mu}\) z\quad\forall\,z\in(0,\infty)\;.
\]

\medskip In the simplified Fermi--Dirac (sFD) model one has
\[\label{Eqn:RTM}
R(z)=\frac1{\frac1z+\frac\eta{z^{1/3}}}\quad\forall\,z\in(0,\infty)\;,
\]
which shares the same asymptotic as the function $R$ in the (FD) case. The constants $\eta$ and $\mu$ are related by \eqref{Eqn:EtaMu} so that the function $H(z)=\log z+ \frac32\,\eta\,z^{2/3}$ has the same behavior as in the case (FD) as $z\to\pm\infty$.

\medskip Finally, let us conclude with a useful estimate.
\begin{lemma}\label{Lem:Equivalents} In the (sFD) case, we have
\[
0\le\frac{z-R(z)}{z^{5/3}}\le\eta\quad\forall\,z\in(0,\infty)\;.
\]
In the (FD) case, there exists a positive constant $C$, independent of $\eta$ such that
\[
0\le\frac{z-R(z)}{z^{5/3}}\le C\,\eta\quad\forall\,z\in(0,\infty)\;.
\]
\end{lemma}
\begin{proof} A straightforward computation shows that we have $z-R(z)=\frac{\eta\,z^{5/3}}{1+\eta\,z^{2/3}}$ in the (sFD) case. The upper bound in the (FD) case follows by considering equivalents as $z\to\infty$ with $w=z\,\eta^{3/2}$ and
\[
C=\sup_{w\in [0,\infty)}\frac1{w^{5/3}}\,\left[
w-\frac1{\sqrt6}\,f_{-1/2}\(f_{1/2}\(\tfrac12\,\sqrt6\,w\)\)\right]\,.
\]
As for the lower bound, we observe that
\begin{eqnarray*}
f_{1/2}(z)&=&\int_0^\infty\frac{\sqrt x}{1+\exp\,(x-z)}\;dx=\int_0^\infty\sqrt x\,\frac{1+\exp\,(x-z)}{(1+\exp\,(x-z))^2}\;dx\\
&\ge&\int_0^\infty\sqrt x\,\frac{\exp\,(x-z)}{(1+\exp\,(x-z))^2}\;dx\\
&&\quad=\left[-\,\frac{\sqrt x}{1+\exp\,(x-z)}\right]_0^\infty+\frac12\int_0^\infty\frac{x^{-1/2}}{1+\exp\,(x-z)}\;dx=\frac 12\,f_{-1/2}(z)
\end{eqnarray*}
where the last line follows from an integration by parts. Coming back to the expression of $R$, $R(z)\le z$ is equivalent to $f_{-1/2}(t)\le2\,f_{1/2}(t)$ with $t=f_{1/2}^{-1}(2z/\mu)$, which is precisely what we have just shown.\end{proof}

\bigskip\noindent{\footnotesize\copyright~2013 by the authors. This paper may be reproduced, in its entirety, for non-commercial purposes.}



\begin{thebibliography}{10}

\bibitem{MR2065020}
{\sc A.~Arnold, J.~A. Carrillo, L.~Desvillettes, J.~Dolbeault, A.~J{\"u}ngel,
  C.~Lederman, P.~A. Markowich, G.~Toscani, and C.~Villani}, {\em Entropies and
  equilibria of many-particle systems: an essay on recent research}, Monatsh.
  Math., 142 (2004), pp.~35--43.

\bibitem{BDEMN}
{\sc P.~Biler, J.~Dolbeault, M.~Esteban, T.~Nadzieja, and P.~Markowich}, {\em
  Steady states for {S}treater's energy-transport models of self-gravitating
  particles.}, in Transport in Transition Regimes, vol.~135 of IMA Volumes in
  Mathematics and Its Applications, Springer, Warsaw, 2004, pp.~37--56.

\bibitem{MR1305221}
{\sc P.~Biler, D.~Hilhorst, and T.~Nadzieja}, {\em Existence and nonexistence
  of solutions for a model of gravitational interaction of particles. {II}},
  Colloq. Math., 67 (1994), pp.~297--308.

\bibitem{MR2099972}
{\sc P.~Biler, P.~Lauren{\c{c}}ot, and T.~Nadzieja}, {\em On an evolution
  system describing self-gravitating {F}ermi-{D}irac particles}, Adv.
  Differential Equations, 9 (2004), pp.~563--586.

\bibitem{MR2143357}
{\sc P.~Biler, T.~Nadzieja, and R.~Sta{\'n}czy}, {\em Nonisothermal systems of
  self-attracting {F}ermi-{D}irac particles}, in Nonlocal elliptic and
  parabolic problems, vol.~66 of Banach Center Publ., Polish Acad. Sci.,
  Warsaw, 2004, pp.~61--78.

\bibitem{biler2004parabolic}
{\sc P.~Biler and R.~Sta{\'n}czy}, {\em {Parabolic-elliptic systems with
  general density-pressure relations}}, S\=urikaisekikenky\=usho K\=oky\=uroku,
  1405 (2004), pp.~31--53.

\bibitem{MR2675439}
\leavevmode\vrule height 2pt depth -1.6pt width 23pt, {\em Mean field models
  for self-gravitating particles}, Folia Math., 13 (2006), pp.~3--19.

\bibitem{MR2305349}
\leavevmode\vrule height 2pt depth -1.6pt width 23pt, {\em Nonlinear diffusion
  models for self-gravitating particles}, in Free boundary problems, vol.~154
  of Internat. Ser. Numer. Math., Birkh\"auser, Basel, 2007, pp.~107--116.

\bibitem{MR1853037}
{\sc J.~A. Carrillo, A.~J{\"u}ngel, P.~A. Markowich, G.~Toscani, and
  A.~Unterreiter}, {\em Entropy dissipation methods for degenerate parabolic
  problems and generalized {S}obolev inequalities}, Monatsh. Math., 133 (2001),
  pp.~1--82.

\bibitem{C}
{\sc P.-H. Chavanis}, {\em Phase transitions in self-gravitating systems},
  International Journal of Modern Physics B, 20 (2006), pp.~3113--3198.

\bibitem{MR2092680}
{\sc P.-H. Chavanis, P.~Lauren{\c{c}}ot, and M.~Lemou}, {\em Chapman-{E}nskog
  derivation of the generalized {S}moluchowski equation}, Phys. A, 341 (2004),
  pp.~145--164.

\bibitem{CSR}
{\sc P.-H. Chavanis, J.~Sommeria, and R.~Robert}, {\em Statistical mechanics of
  two-dimensional vortices and collisionless stellar systems}, Astrophys. J.,
  471 (1996), p.~385.

\bibitem{MR2338354}
{\sc J.~Dolbeault, P.~Markowich, D.~Oelz, and C.~Schmeiser}, {\em Non linear
  diffusions as limit of kinetic equations with relaxation collision kernels},
  Arch. Ration. Mech. Anal., 186 (2007), pp.~133--158.

\bibitem{DOLBEAULT:2009:HAL-00349574:2}
{\sc J.~{D}olbeault and R.~{S}ta\'nczy}, {\em {N}on-existence and uniqueness
  results for supercritical semilinear elliptic equations}, {A}nnales {H}enri
  {P}oincar{\'e}, 10 (2009), pp.~1311--1333.

\bibitem{MR544879}
{\sc B.~Gidas, W.~M. Ni, and L.~Nirenberg}, {\em Symmetry and related
  properties via the maximum principle}, Comm. Math. Phys., 68 (1979),
  pp.~209--243.

\bibitem{MR0340701}
{\sc D.~D. Joseph and T.~S. Lundgren}, {\em Quasilinear {D}irichlet problems
  driven by positive sources}, Arch. Rational Mech. Anal., 49 (1972/73),
  pp.~241--269.

\bibitem{MR1842429}
{\sc F.~Otto}, {\em The geometry of dissipative evolution equations: the porous
  medium equation}, Comm. Partial Differential Equations, 26 (2001),
  pp.~101--174.

\bibitem{MR2136979}
{\sc R.~Sta{\'n}czy}, {\em Steady states for a system describing
  self-gravitating {F}ermi-{D}irac particles}, Differential Integral Equations,
  18 (2005), pp.~567--582.

\bibitem{MR2295189}
\leavevmode\vrule height 2pt depth -1.6pt width 23pt, {\em On some
  parabolic-elliptic system with self-similar pressure term}, in Self-similar
  solutions of nonlinear {PDE}, vol.~74 of Banach Center Publ., Polish Acad.
  Sci., Warsaw, 2006, pp.~205--215.

\bibitem{Stanczy07}
\leavevmode\vrule height 2pt depth -1.6pt width 23pt, {\em Reaction-diffusion
  equations with nonlocal term}, in Equadiff 2007, Wien, 2007.

\bibitem{MR2548877}
\leavevmode\vrule height 2pt depth -1.6pt width 23pt, {\em Stationary solutions
  of the generalized {S}moluchowski-{P}oisson equation}, in Parabolic and
  {N}avier-{S}tokes equations. {P}art 2, vol.~81 of Banach Center Publ., Polish
  Acad. Sci. Inst. Math., Warsaw, 2008, pp.~493--500.

\bibitem{MR2506790}
\leavevmode\vrule height 2pt depth -1.6pt width 23pt, {\em The existence of
  equilibria of many-particle systems}, Proc. Roy. Soc. Edinburgh Sect. A, 139
  (2009), pp.~623--631.

\bibitem{stanczy-evolution}
\leavevmode\vrule height 2pt depth -1.6pt width 23pt, {\em On an evolution
  system describing self-gravitating particles in microcanonical setting},
  Monatshefte f\"ur Mathematik, 162 (2011), pp.~197--224.

\bibitem{Wol}
{\sc G.~Wolansky}, {\em Critical behaviour of semi-linear elliptic equations
  with sub-critical exponents}, Nonlinear Analysis, 26 (1996), pp.~971--995.

\end{thebibliography}
\end{document}